\documentclass[12pt,letterpaper]{amsart}

\usepackage{amsmath,amssymb,amsthm}
\usepackage{fullpage}
\usepackage{verbatim}
\usepackage[noadjust]{cite} 
\usepackage{hyperref}
\usepackage[capitalize, nameinlink]{cleveref}
\usepackage{enumitem}
\usepackage[usenames,dvipsnames]{xcolor} 
\usepackage{tikz}
\usepackage[font=small,labelfont=bf, margin=.5in]{caption}
\usepackage{babel}
\usepackage{parskip, amsfonts, amsmath, changepage, amsthm, setspace, geometry, verbatim, MnSymbol,chemarrow, tikz}

\colorlet{prettygreen}{ForestGreen!60!LimeGreen}

\makeatletter
\def\@defaultbiblabelstyle#1{[#1]}
\makeatother

\usetikzlibrary{calc, arrows.meta}
\tikzset{vtx/.style={circle, fill, inner sep=1.5pt}}
\tikzset{openvtx/.style={circle, draw, inner sep=1.5pt}}

\usepackage{tikz-cd}

\newcommand{\eps}{\epsilon}

\newcommand{\Z}{{\mathbb Z}}

\title{Chromatic numbers of rank-two Abelian Cayley graphs}
\author{Mike Krebs}
\address{Department of Mathematics, California State University --- Los
Angeles}
\email{mkrebs@calstatela.edu}
\author{Alejandro Leyva}
\address{Department of Mathematics, University of California, Berkeley}
\email{aleyva07@berkeley.edu}

\makeatletter
\let\mytitle\@title
\let\myauthor\@author
\makeatother

\newtheorem{Thm}{Theorem}[section]

\newtheorem{Lem}[Thm]{Lemma}
\newtheorem{Cor}[Thm]{Corollary}
\newtheorem{Guess}[Thm]{Conjecture}
 
\theoremstyle{definition}
\newtheorem{Def}[Thm]{Definition}
 
\theoremstyle{remark}

\begin{document}

\begin{abstract} A connected Cayley graph for an Abelian group generated by a finite symmetric subset $S$ can be represented by an integer matrix, its Heuberger matrix.  We call the number of columns of that matrix its \emph{rank} and the number of rows its \emph{dimension}.  Several previous papers have dealt with the question of finding a formula for the chromatic number of an Abelian Cayley graph in terms of an associated Heuberger matrix.  In this paper, we fully resolve this matter for all integer matrices of rank $\leq 2$.  Prior results provide such formulas when the rank is 1, as well as when the rank is 2 and the dimension is no more than 4.  Here, we complete the picture for the rank-two case by showing that when the rank is 2 and the dimension is at least 5, then the chromatic number equals 3 unless the graph has loops (in which case it is uncolorable); the graph is bipartite (in which case the chromatic number is 2); or the matrix has a zero row (in which case, the chromatic number does not change when that row is deleted).
\end{abstract}


\maketitle
 
\section{Introduction}

\subsection{Overview}
We begin by recalling some basic definitions.  We say a subset $S$ of a group $G$ is \emph{symmetric} if $x^{-1}\in S$ whenever $x\in S$.  The \emph{Cayley graph} for a group $G$ with a symmetric subset $S$, denoted $\text{Cay}(G,S)$, is the graph whose vertices are the elements of $G$, where each vertex $g$ is adjacent to every vertex of the form $gs$ with $s\in S$.  In this case we call the set $S$ the \emph{generating set}, and we refer to its elements as \emph{generators}.  When the group $G$ is abelian, naturally enough we call the graph an \emph{Abelian Cayley graph}.  The \emph{chromatic number} of a graph $X$, denoted $\chi(X)$, is the minimum number of colors needed to assign each vertex in the graph a color such that no two adjacent vertices have the same color.

The chromatic number of Cayley graphs is important in information theory (see, for example, \cite{Info1}) as well as in the construction of expanders and concentrators (see, for example, \cite{Info2}). There are also applications with random sets and graphs, as seen in \cite{RandomGraph} and \cite{RandomSet}.

This paper focuses on chromatic numbers of Abelian Cayley graphs.  As discussed in \cite{Cervantes-matrix-method}, to every finite-degree, connected Abelian Cayley graph we may associate an integer matrix called a \emph{Heuberger matrix}.  Conversely, every integer matrix has an associated Abelian Cayley graph.  In \cref{subsec:Prelims} we lay out the essentials of this correspondence.



We call the number of columns of the Heuberger matrix its \emph{rank} and the the number of rows its \emph{dimension}.  The goal of this research program is to obtain the chromatic number of an Abelian Cayley graph from an associated Heuberger matrix.  We may divvy into cases according to the rank and dimension.

In matrices where the rank is greater than the dimension, one of the columns will be in the $\Z$-span of the others, so it can be reduced to a square matrix that represents the same graph.  (See \cref{lem:isomorphisms} for more detail.)  So, without loss of generality, we can assume that the number of rows is at least the number of columns.


\subsubsection{The past} Previously, chromatic numbers of Abelian Cayley graphs have been found in cases where the rank and/or dimension is small.  Specifically, for rank 1 the chromatic number is determined in \cite{Cervantes-matrix-method}.  (See also \cite{Carrillo-et-al}.)  We state that result here as \cref{thm:tomato}.  For rank 2, dimensions 2 and 3, theorems completely determining the chromatic number appear in \cite{small-dim-rank}.  These are stated here as \cref{thm:two-by-two} and \cref{3by2}.  For rank 2, dimension 4, this is accomplished in \cite{Eng}.  That theorem is stated here as \cref{thm:four-by-two}.


\subsubsection{The present} The main new result in this paper is for $m\times 2$ matrices where $m\geq 5$. In these cases, the chromatic number of the associated Cayley graph is $3$, with a few standard exceptions.  Recall that a \emph{loop} in a graph is an edge from a vertex to itself.  A graph with a loop cannot be properly colored.  (We adopt the convention that the chromatic number of a graph with loops is $\infty$.)  A graph is \emph{bipartite} if its vertex set can be partitioned into two sets $A$ and $B$ such that every edge has one endpoint in $A$ and one in $B$.  A connected bipartite graph always has chromatic number is $2$.  We call a row of a matrix a \emph{zero row} if all of its entries are zero.

\begin{Thm}\label{thm:main}
    Let $m\geq 5$ be an integer.  Let $M_X$ be an $m\times 2$ integer matrix and $X=M_X^{SACG}$ (the Abelian Cayley graph associated with $M_X$). Suppose $M_X$ has no zero rows.  If $X$ is not bipartite and has no loops, then $\chi(X)=3$.
\end{Thm}

We remark that the case when $M_X$ has a zero row is handled by \cref{lem:delete-zero-row}, which ensures that we may delete all zero rows without changing the chromatic number, whereupon a theorem for a smaller number of rows (i.e., one of Theorems \ref{thm:main}, \ref{thm:two-by-two}, \ref{3by2}, or \ref{thm:four-by-two}) can be applied.


\subsubsection{The future} This paper settles our main question for matrices of rank $2$, so one obvious direction for future work is in matrices with rank greater than 2. For rank 3, the square matrices include circulant graphs of degree six. While there are many partial results, a complete formula for the chromatic numbers of such circulant graphs is unknown.

\cref{thm:main} tells us that, for the rank-two case, the behavior of the chromatic number stabilizes starting at dimension $5$.  An inspection of the method of proof suggests that something similar may occur for higher-rank cases as well.  The proofs of Theorems \ref{3by2}, \ref{thm:four-by-two} and \ref{thm:main} all stem from a single central idea.  Namely, we obtain a graph homomorphism from a graph with an $m\times r$ Heuberger matrix to one with an $(m-1)\times r$ Heuberger matrix whenever we merge two rows by adding or subtracting them --- see \cref{lem:homomorphisms}.  In this way we obtain upper bounds on chromatic numbers of graphs with larger matrices from those with smaller matrices.  There are $2{m\choose 2}$ such homomorphisms.  With each homomorphism comes the potential to establish an upper bound of $3$ on the chromatic number of the original graph.  Fixing $r$ and letting $m$ increase, we have more and more such potential.  We therefore offer the following conjecture:

\begin{Guess}\label{conjecture}
Let $r$ be a positive integer.  Then for all sufficiently large $m$, if $M_X$ is an $m\times r$ integer matrix and $X=M_X^{SACG}$ (the Abelian Cayley graph associated with $M_X$), and $M_X$ has no zero rows and $X$ is not bipartite and has no loops, then $\chi(X)=3$.\end{Guess}

We remark that \cref{conjecture} holds for both $r=1$ and $r=2$, as shown in \cref{thm:tomato} and \cref{thm:main}, respectively.

\subsection{Preliminaries} \label{subsec:Prelims}
In this section we quickly sketch the basic correspondence between Abelian Cayley graphs and integer matrices.  For more details, see \cite{Cervantes-matrix-method}.

Let $G$ be a finitely generated Abelian group, and let $S=\{ \pm s_1, \pm s_2, \dots,\pm s_m \}$ be a symmetric subset of $G$ that generates $G$.  (In this paper, for Abelian groups we always use additive notation.) Let $e_k$ denote the $k$th standard basis vector in $\Z_m$, that is, the vector whose $k$th component is $1$ and with every other component equal to $0$.  We define a homomorphism $\phi:\Z^m\to G$ where $e_k\mapsto s_k$.
If $H=\ker\phi$, then $\phi$ induces a natural graph homomorphism from $X=Cay(\Z^m/H,\{ H\pm e_1, \dots, H\pm e_m\})$ to $Cay(G,S)$. Suppose that we have $y_1,\dots,y_r\in \Z^m$ such that $H=\langle y_1,\dots,y_r\rangle$. 
Then, to represent the graph $X$ we can write the $m\times r$ matrix $M_X$ where the $j$th column is $y_j$. We say that $M_X$ is a (non-unique) \emph{Heuberger matrix} of $X$. This construction can also go the other way: given any integer matrix $M$, there exists an Abelian Cayley graph $X$ with $M$ as a Heuberger matrix.  In this case we write $X=M^{SACG}$, and we refer to $X$ as the SACG associated to $M$.  (SACG stands for ``standardized Abelian Cayley graph.'')

\begin{Lem}\label{lem:loops}
Let $M$ be an integer matrix, and let $X=M^\text{SACG}$.  Then $X$ has loops if and only if for some $i$ the standard basis vector $e_i$ is in the $\Z$-span of the columns of $M$.
\end{Lem}

\begin{Lem}\label{lem:isomorphisms}
    Let $X$ and $X'$ be Abelian Cayley graphs with Heuberger matrices $M_X$ and $M_{X'}$, respectively.
    \begin{enumerate}
        \item If $M_{X'}$ is obtained by permuting columns of $M_X$, then $X=X'$.
        \item If $M_{X'}$ is obtained by multiplying a column of $M_X$ by $-1$, then $X=X'$.
        \item If $M_{X'}$ is obtained by replacing the $j$th column of $M_X$ with $y_j+ay_i$ ($i\not= j)$ for some $a\in \Z$, then $X=X'$.
        \item If $M_{X'}$ is obtained by deleting a column of $M_X$ that is in the $\Z$-span of the other columns, then $X=X'$. 
        \item If $M_{X'}$ is obtained by permuting rows of $M_X$, then $X$ is isomorphic to $X'$.
        \item If $M_{X'}$ is obtained by multiplying a row of $M_X$ by $-1$, then $X$ is isomorphic to $X'$.
    \end{enumerate}
\end{Lem}

For integer matrices $M_X$ and $M_{X'}$, we write $M_X\cong M_{X'}$ to indicate that their associated SACGs are isomorphic as graphs.

An integer matrix is \emph{unimodular} if it is square and has determinant $\pm 1$.  A \emph{signed permutation matrix} is a square matrix such that for every row and for every column, there is exactly one nonzero entry, which is $\pm 1$.  Beginning with an integer matrix $M_X$ and performing a finite sequence of actions as in (1), (2), and (3) in \cref{lem:isomorphisms} is equivalent to multiplying $M_X$ on the right by a unimodular matrix.  Beginning with an integer matrix $M_X$ and performing a finite sequence of actions as in (5) and (6) in \cref{lem:isomorphisms} is equivalent to multiplying $M_X$ on the left by a signed permutation matrix.  

Recall that a \emph{graph homomorphism} is a function $f$ from the vertex set of a graph $X$ to the vertex set of a graph $Y$ such that if $v$ and $w$ are adjacent in $X$, then $f(v)$ and $f(w)$ are adjacent in $Y$.  The following lemma is well-known.

\begin{Lem}\label{lem:general-upper-bound}
    If there is a graph homomorphism from $X$ to $Y$, then $\chi(X)\leq \chi(Y)$.
\end{Lem}

We remark that \cref{lem:general-upper-bound} applies even in this case where $Y$ has loops; but then we get that $\chi(X)\leq \infty$, which furnishes no useful information.

We now provide some standard graph homomorphisms between Abelian Cayley graphs.

\begin{Lem}\label{lem:homomorphisms}
    The following operations on integer matrices induce graph homomorphisms on the corresponding Abelian Cayley graphs:
    \begin{enumerate}
        \item Any of the transformations from \cref{lem:isomorphisms}
        \item Reducing a column by a common integer factor
        \item Collapsing the top two rows by adding them
        \item Appending an arbitrary column
        \item Appending a zero row
        \item Any composition of the above
    \end{enumerate}
\end{Lem}

\begin{Lem}\label{lem:delete-zero-row}
Let $M_X$ and $M_{X'}$ be integer matrices with associated SACGs $X$ and $X'$, respectively.  Suppose that $M_{X'}$ is obtained from $M_X$ by deleting a zero row.  Then $\chi(M_X)=\chi(M_Y)$.
\end{Lem}

In many cases, we indicate the existence of a homomorphism from $X$ to $Y$ as in Lemma \ref{lem:homomorphisms} without explicitly writing it out. In such cases we write $M_X\xrightarrow{\ocirc} M_Y$, where $M_X$ and $M_Y$ are Heuberger matrices of $X$ and $Y$, respectively.  One common situation in which we do so occurs when we either add or subtract two rows. In this case, we will use $\varepsilon_i$ (for some index $i$) to represent either $1$ or $-1$. For example,
$$\begin{pmatrix}
    x_{11}&x_{12}\\
    x_{21} & x_{22}
\end{pmatrix}\xrightarrow{\ocirc} \begin{pmatrix}
    x_{11}+\varepsilon_1x_{21}&x_{12}+\varepsilon_1x_{22}
\end{pmatrix}$$ describes a homomorphism where we either add the first and second row or multiply the second row by $-1$ and then add the rows.

The next few results describe the chromatic number for very simple Abelian Cayley graphs in terms of Heuberger matrices.
\begin{Lem}
    Let $M$ be an integer matrix, and let $X=M^{SACG}$.  Then $X$ is bipartite if and only if all column sums of $M$ are even.
\end{Lem}
\begin{Thm}[Tomato Cage Theorem]\label{thm:tomato}
    Let $y_1=(y_{11},\dots, y_{m1})^T\in\Z^m$ and let $X$ be the SACG defined by $(y_1)_X^{SACG}$. If $y_1=\pm e_i$ for some $i$, then $X$ has loops and cannot be properly colored. If that's not the case, then
    $$\chi(X)=\begin{cases}
        2&\text{if } y_1 \text{ has an even number of odd entries} \\
        3&\text{if } y_1 \text{ has an odd number of odd entries}
    \end{cases}$$
\end{Thm}

\section{Previous Results}

    \subsection{$1\times 2$ Heuberger matrices}
    The result for a $1\times 2$ matrix follows from a stronger result from \cite{small-dim-rank}. 
    \begin{Lem}
        Let $y_1,\dots,y_r$ be integers, not all $0$. Suppose $X$ is a SACG defined by $(y_1,\dots,y_r)^{SACG}_X$. Let $e=\gcd(y_1,\dots,y_r)$ Then $X$ has loops if and only if $e=1$. If there are no loops, then $\chi(X)=2$ if $e$ is even and $\chi(X)=3$ if $e$ is odd.
    \end{Lem}
    
    \subsection{$2\times 2$ Heuberger matrices}

    \begin{Def}
        Let $a,b,n$ be relatively prime integers where $n\neq 0$, $n\nmid a$, and $n\nmid b$. then, the graph $\text{Cay}(\Z_n,\{ \pm a,\pm b \})$ is a \emph{Heuberger circulant}, written as $C_n(a,b)$. 
    \end{Def}
    The relative primeness is so the graph is connected and the condition where $n\nmid a,b$ is so there are no loops.
    
    The chromatic numbers for Heuberger circulants was determined by Heuberger:
    \begin{Thm}[{\cite[Theorem 3]{HEUBERGER2003153}}]
        Let $C_n(a,b)$ be a Heuberger circulant. 
        $$\chi(C_n(a,b))= \begin{cases}
            2 \text{ if } a,b \text{ are both odd, but } n \text{ is even} \\
            5  \text{ if } n=\pm 5 \text{ and } a\equiv \pm 2b \mod 5\\
            4 \text{ if } n=\pm13 \text{ and }a\equiv \pm 5b \mod 13 \\
            4 \text{ if } (i) \ n\neq \pm 5 \text{ and } (ii) \ 3\nmid n \text{ and } (iii)\  a\equiv \pm 2b \mod n \text{ or } b\equiv \pm 2a \mod n \\
            3 \text{ otherwise}
        \end{cases}$$

    \end{Thm}

The following comes from \cite{small-dim-rank}: Suppose $M_X$ is a $2\times 2$ integer matrix, and that $X=M_X^{SACG}$.  We can always perform row and column operations as in \cref{lem:isomorphisms} to create a lower-triangular $2\times 2$ matrix $M_{X'}$ with non-negative diagonal entries such that $X'=M_{X'}^{SACG}$ is isomorphic to $X$.

     With this, the chromatic number for any $2\times 2$ matrix $M_X$ can be computed.
    
    \begin{Thm}[\cite{small-dim-rank}]\label{thm:two-by-two}
        Let $X$ be an SACG defined by
        $$\begin{pmatrix}
            y_{11}&0\\
            y_{21} & y_{22}
        \end{pmatrix}_X^{SACG}$$
        Also, suppose that $y_{11}\geq 0$ and $y_{22}\geq 0$. Let $d=\gcd(y_{11},y_{21})$ and $e=\gcd(y_{11},y_{21},y_{22})$. 
        \begin{enumerate}
            \item If either $(i)$ $y_{22}=1$ or $(ii)$ $y_{11}=1$ and $y_{22}|y_{21}$ or $(iii)$ $y_{11}=0$ and $\gcd(y_{21},y_{22})=1$, then $X$ has loops and is not colorable.
            \item If both $y_{11}+y_{21}$ and $y_{22}$ are even, then $\chi(X)=2$
            \item If $(i)$ neither of the above holds and $(ii)$ $y_{11}=0$ or $y_{22}=0$ or $e>1$ or $y_{22}|y_{21}$, then $\chi(X)=3$.
            \item If none of the above hold, take $q\in \Z$ such that $\gcd(y_{11},y_{21}+qy_{22})=1$. Then, $\chi(X)=\chi(C_n(a,b))$, where $a=-y_{21}-qy_{22}$, $b=y_{11}$, and $n=y_{11}y_{22}$.
        \end{enumerate}
    \end{Thm}

    \subsection{$3\times 2$ Heuberger matrices}

We begin with a technical lemma about $3\times 2$ matrices that we'll use later in the proof of the main theorem.

\begin{Lem}\label{lemma-L-shaped}Suppose we have $X=M_X^{SACG}$, where \[M_X=\begin{pmatrix}
    y_{11} && 0 \\
    y_{21} && 0 \\
    y_{31} && y_{32} 
    \end{pmatrix},\]  
    
    and that $y_{11},y_{21},y_{32}>0$ and $-\frac{y_{32}}{2}\leq y_{31}\leq 0$.
    
    Then:
    
\begin{enumerate}

    \item We have that $X$ has loops if and only if $y_{32}=1$.
    
    \item We have that $\chi(X)=2$ if and only if $y_{11}+y_{21}+y_{31}$ and $y_{32}$ are both even.
    
    \item We have that $\chi(X)=4$ if and only if $y_{11}=y_{21}=-y_{31}=1$ and $3\nmid y_{32}$ and $y_{32}>1$.
    
    \item Otherwise, $\chi(X)=3$.

\end{enumerate}
\end{Lem}

We now turn our attention to the statement of the theorem which gives the chromatic number for graphs associated with arbitrary $3\times 2$ integer matrices.

    \begin{Def}
        Let $$M=\begin{pmatrix}
            y_{11}&y_{12} \\
            y_{21}& y_{22}\\
            y_{31}&y_{32}
        \end{pmatrix}$$ be an integer matrix with no zero rows. Then, $M$ is in \emph{modified Hermite Normal Form} (MHNF) if all the following hold:
        \begin{enumerate}
            \item $y_{11}>0$
            \item $y_{12}=0$
            \item $y_{11}y_{22}\equiv y_{11}y_{32}\mod 3$ (This is equivalent to $y_{11}\equiv 0$ or $y_{22}\equiv y_{32}\mod 3$)
            \item $y_{22}\leq y_{32}$
            \item $|y_{22}|\leq |y_{32}|$
            \item Either:
            \begin{enumerate}
                \item $y_{22}=0$ and $-\frac{1}{2}|y_{32}|\leq y_{31}\leq 0$ or
                \item $-\frac{1}{2}|y_{22}|\leq y_{21}\leq 0$
            \end{enumerate}
        \end{enumerate}
    \end{Def}
    
    Lemma 2.12 in \cite{small-dim-rank} gives a step-by-step method to turn any $3\times 2$ integer matrix $M_X$ without zero rows into a matrix $M_{X'}$ in MHNF in such a way that $M_X^{SACG}$ is isomorphic to $M_{X'}^{SACG}$.  We have a formula for the chromatic number of a $3\times 2$ matrix in MHNF, namely:
    \begin{Thm}[\cite{small-dim-rank}]\label{3by2}
        Let $X$ be a standardized Abelian Cayley graph with Heuberger matrix $$
        M_X=\begin{pmatrix}
            y_{11}&0\\
            y_{21}&y_{22}\\
            y_{31}&y_{32}
        \end{pmatrix}$$ in MHNF. 
        \begin{enumerate}
            \item If the first column is $e_1$ or the second column is $e_3$, then $X$ has loops and cannot be colored.
            \item If $y_{11}+y_{21}+y_{31}$ and $y_{22}+y_{32}$ are both even, then $\chi(X)=2$.
            \item If $$M_X=\begin{pmatrix}
                1&0\\0&1\\\pm3k&1+3k
            \end{pmatrix} \text{ or } \begin{pmatrix}
                1&0\\0&-1\\\pm3k&-1+3k
            \end{pmatrix}\text{ or } \begin{pmatrix}
                1&0\\-1&2\\-1-3k&2+3k
            \end{pmatrix}\text{ or } \begin{pmatrix}
                1&0\\-1&-2\\-1+3k&-2+3k
            \end{pmatrix}$$ for some $k\in \Z^{+}$ or $M_X=\begin{pmatrix}
                1&0\\0&-1\\3b&2
            \end{pmatrix}$ for some $b\in \Z$, or $M_X=\begin{pmatrix}
                1&0\\-1&a\\-1&a+3(k-1)
            \end{pmatrix}$ for $a\in\Z$ where $3\nmid a$ and some $k\in \Z^{+}$, then $\chi(X)=4$.
            \item If none of the above hold, then $\chi(X)=3$.
        \end{enumerate}
    \end{Thm}
    
We refer to the various possibilities in \cref{3by2}(3) as the \emph{six exceptional cases}.

    We say that a row of an integer matrix is \emph{divisible by $n$} if all elements of the row are divisible by $n$.  If a matrix has a row divisible by $3$, then that property is preserved by the various row and column operations in \cref{lem:isomorphisms}.  None of the six exceptional cases in \cref{3by2} have a row divisible by $3$.  This yields the following corollary, which we will make use of in Section \ref{sec:main-theorem-proof}.

    \begin{Cor}\label{cor:3-by-2-three-div-row}
        Let $M_X$ be a $3\times 2$ integer matrix, and let $X=M_X^{SACG}$. Suppose that $M_X$ has no zero rows and that $X$ does not have loops. If some row of $M_X$ is divisible by 3, then $\chi(X)\leq 3$.
    \end{Cor}  
    
    \subsection{$4\times 2$ Heuberger matrices}

Before stating the main theorem for chromatic numbers of Abelian Cayley graphs with $4\times 2$ Heuberger matrices, we first state a special case, which will be useful in the proof of this paper's main theorem in Section \ref{sec:main-theorem-proof}.

    \begin{Lem}[\cite{Eng}]\label{lem:4-by-2-three-div-row}
        Let $M_Y$ be a $4\times 2$ integer matrix, and let $Y=M_Y^{SACG}$. Suppose that $M_Y$ has no zero rows and that $Y$ does not have loops. If some row of $M_Y$ is divisible by 3, then $\chi(Y)\leq 3$.
    \end{Lem}

Observe the similarity between \cref{cor:3-by-2-three-div-row} and \cref{lem:4-by-2-three-div-row}.  We refer to them as the $3\times 2$ (respectively, $4\times 2$) three-divisible row lemmas.

    \begin{Thm}[\cite{Eng}]\label{thm:four-by-two}
        Let $M$ be a $4\times 2$ integer matrix, and let $X=M^{\text{SACG}}$. Furthermore suppose that $X$ is not bipartite; that $X$ has no loops; and that $M$ has no zero rows. Then $\chi(X)=4$ if and only if there exists a signed permutation matrix $P$ and unimodular matrix $U$ such that  
        $$PMU=\begin{pmatrix}
            1&a\\1&b\\1&c\\0&1
        \end{pmatrix}$$
        for $a,b,c\in \Z$ with $3\mid a+b+c$.  Otherwise, $\chi(X)=3$. 
    \end{Thm}

\section{Proof of \cref{thm:main}}\label{sec:main-theorem-proof}

In this section we prove \cref{thm:main}.  We begin with the case where $m=5$.  This will then serve as the base of an induction that proves the theorem for all $m\geq 5$.

    \subsection{$5\times 2$ Heuberger matrices}
    The main idea of the proof for $m=5$ is to use \cref{lem:homomorphisms}(3) to obtain a graph homomorphism from our graph with a given $5\times 2$ matrix to a graph with a matrix with fewer rows.  Then we can take advantage of the main theorems for $m\times 2$ matrices with $m<5$.  In particular, \cref{lem:4-by-2-three-div-row} will set things in motion.  To make sure the conditions of that result are met, we first require a lemma.  Recall our convention that every variable $\epsilon_i$ equals $\pm 1$.

    \begin{Lem}\label{rowdiv3}
    In an arbitrary $5\times 2$ integer matrix $M_Z$, then either some row of $M_Z$ is divisible by $3$, or there exist two distinct rows $r_1,r_2$ of $M_Z$ such that $r_1 +\epsilon_1r_2$ is divisible by 3.
\end{Lem}

\begin{proof}
    Suppose that no row of $M_Z$ is divisible by 3.  Let $r_1, r_2, r_3, r_4, r_5$ be the rows of $M_Z$.  By a harmless (we hope) abuse of notation, also let $r_i$ denote the reduction mod 3 of the $i$th row of $M_Z$.  So $r_1, r_2, r_3, r_4, r_5$ are elements of $V=\mathbb{Z}_3^2\setminus\{(0,0)\}$, which has eight elements.  Partition $V$ into four sets, where each set contains a vector and its negative.  That is, the four sets are \[\{(1,0),(2,0)\}, \{(0,1),(0,2)\} , \{(1,1),(2,2)\} , \{(1,2),(2,1)\}.\]  By Pigeonhole Principle, two of the five elements $r_1, r_2, r_3, r_4, r_5$ must belong to the same set in this partition.  Therefore two rows are either equal mod 3 (in which case, let $\epsilon_1=-1$) or else opposite mod 3 (in which case, let $\epsilon_1=1$).
\end{proof}

\begin{Thm}\label{thm:base-case-for-main-thm}
    Let $M_X$ be a $5\times 2$ integer matrix and $X=M_X^{\text{SACG}}$. Suppose that $X$ is not bipartite; that $X$ has no loops; and that $M_X$ has no zero rows. Then $\chi(X)=3$.
\end{Thm}
\begin{proof}
Let 
\[M_X=\begin{pmatrix}
    x_{11} & x_{12}\\
    x_{21} & x_{22}\\
    x_{31} & x_{32}\\
    x_{41} & x_{42}\\
    x_{51} & x_{52}
    \end{pmatrix}\]
and assume that $\chi(X)\geq 4$.  We claim that two rows of $M_X$ are linearly dependent over $\mathbb{Q}$.

Without loss of generality, we can use \cref{rowdiv3} to collapse the first two rows to get $M_X\xrightarrow{\ocirc} M_Y$ where some row of $M_Y$ is divisible by 3.  (That is: If some row of $M_X$ is divisible by $3$, then permute rows to move it to the third row, and then collapse the top two rows by adding them.  Otherwise, take two rows that are either equal or opposite mod $3$, permute rows so as to move them to the top, and then collapse them by either adding or subtracting to produce a row divisible by $3$.)  Let $Y=M_Y^{SACG}$.  By \cref{lem:4-by-2-three-div-row} and \cref{lem:general-upper-bound}, either the graph $Y$ has loops or $M_Y$ has a zero row.  In the latter case, it's immediate that two rows of $M_Y$ are linearly dependent over $\mathbb{Q}$.

So assume $Y$ has loops.  Let $y_1$ be the first column of $M_{Y}$, and let $y_2$ be the second column.  By \cref{lem:loops}, some standard basis vector $e_i$ must lie in the $\Z$-span of $y_1$ and $y_2$.  That is, $\alpha y_1 + \beta y_2=e_i$ for some integers $\alpha, \beta$.  Note that $\alpha, \beta$ are not both zero. But at least two of the bottom three entries of $\alpha y_1 + \beta y_2$ are zero. Further, these bottom three rows of $M_Y$ are also rows of $M_X$.  So those two rows give us a $2\times 2$ minor of $M_X$ with determinant $0$.

We have now shown that two rows of $M_Y$ are linearly dependent over $\mathbb{Q}$.  Assume without loss of generality that these are the bottom two rows.

Therefore, there exist $\alpha,\beta\in\mathbb{Z}$ (not both $0$) such that \[\alpha\begin{pmatrix}x_{41}\\
x_{51}\end{pmatrix}+\beta\begin{pmatrix}x_{42}\\
x_{52}\end{pmatrix}=\begin{pmatrix}0\\0\end{pmatrix}\]
Reducing by a common factor if need be, we can assume that $\alpha$ and $\beta$ are relatively prime, so there exist integers $k,\ell$ such that $\alpha k+\beta\ell=1$.

Multiply $M_X$ on the right by the unimodular matrix \[\begin{pmatrix}
    \alpha & -\ell\\
    \beta & k
\end{pmatrix}\] to get

\[M_X\cong\begin{pmatrix}
    w_{11} & w_{12}\\
    w_{21} & w_{22}\\
    w_{31} & w_{32}\\
    0 & w_{42}\\
    0 & w_{52}
    \end{pmatrix}\cong \begin{pmatrix}
    w_{42} & 0\\
    w_{52} & 0\\
    w_{12} & w_{11}\\
    w_{22} & w_{21}\\
    w_{32} & w_{31}
    \end{pmatrix}\]

for some integers $w_{ij}$.

\textit{Claim:} $|w_{42}|=|w_{52}|=1$.

\textit{Proof of claim:} Consider \[\begin{pmatrix}
    w_{42} & 0\\
    w_{52} & 0\\
    w_{12} & w_{11}\\
    w_{22} & w_{21}\\
    w_{32} & w_{31}
    \end{pmatrix}\xrightarrow{\ocirc}\begin{pmatrix}
    w_{42} & 0\\
    w_{52} & 0\\
    w_{12}+\epsilon_5 w_{22}+\epsilon_6 w_{32} & w_{11}+\epsilon_5 w_{21}+\epsilon_6 w_{31}
    \end{pmatrix}=M_{Q_{\epsilon_5, \epsilon_6}},\]where $Q_{\epsilon_5, \epsilon_6}=(M_{Q_{\epsilon_5, \epsilon_6}})^{SACG}$.  Moreover, by Lemma \ref{lem:isomorphisms} we have that \[M_{Q_{\epsilon_5, \epsilon_6}}\cong\begin{pmatrix}
    |w_{42}| & 0\\
    |w_{52}| & 0\\
    w & |w_{11}+\epsilon_5 w_{21}+\epsilon_6 w_{31}|
    \end{pmatrix}\]for some integer $w$ such that the conditions of \cref{lemma-L-shaped} are satisfied.  Using \cref{lemma-L-shaped}, one can then show that if $Q_{\epsilon_5, \epsilon_6}$ has loops, then $w_{11}+\epsilon_5 w_{21}+\epsilon_6 w_{31}=\pm 1$.  Solving the resulting system of equations, we then find that $Q_{\epsilon_5, \epsilon_6}$ has loops for all choices of $\epsilon_5$ and $\epsilon_6$ if and only if $(w_{11},w_{21},w_{31})^t=e_i$ for some $i=1,2,3$.  But then $X$ has loops, contrary to our starting assumptions.  Hence there is a choice of $\epsilon_5$ and $\epsilon_6$ such that $Q_{\epsilon_5, \epsilon_6}$ does not have loops.  By \cref{lem:general-upper-bound}, we have that $\chi(Q_{\epsilon_5, \epsilon_6})\geq 4$.  One more use of \cref{lemma-L-shaped} then proves the claim.

\vspace{.1in}

Multiplying one or both of the top two rows by $-1$ as needed, we may assume without loss of generality that $w_{42}=w_{52}=1$.  That is:

\[M_X\cong \begin{pmatrix}
    1 & 0\\
    1 & 0\\
    w_{12} & w_{11}\\
    w_{22} & w_{21}\\
    w_{32} & w_{31}
    \end{pmatrix}.\]

From here, we add the top two rows to get:

\[ M_X\xrightarrow{\ocirc} \begin{pmatrix}
    2 & 0\\
    w_{12} & w_{11}\\
    w_{22} & w_{21}\\
    w_{32} & w_{31}
    \end{pmatrix}=M_W.\]

Let $q$ be the number of elements of the set $\{w_{11},w_{21}, w_{31}\}$ that are congruent to $0$ mod $3$.  We now condition on the value of $q$.
\begin{description}
     \item[Case $q=3$] In this case we have that $w_{11}=3t_1, w_{21}=3t_2, w_{31}=3t_3$ for some integers $t_1, t_2, t_3$.  We can collapse the middle two rows to get \[\begin{pmatrix}
    2&0\\w_{12}&3t_1\\w_{22}&3t_2\\w_{32}&3t_3
\end{pmatrix}\xrightarrow{\ocirc}\begin{pmatrix}
    2&0\\w_{12}+w_{22}&3(t_1+t_2)\\w_{32}&3t_3
\end{pmatrix}=M_Z.\] Then $M_Z$ satisfies conditions (1)-(3) of MNHF. We can perform row and column operations on $M_Z$ to obtain a new matrix in MHNF, with the same top row as $M_Z$, whose associated graph is isomorphic to $Z=M_Z^{SACG}$.  Because the top left entry is $2$, this is not one of the six exceptional cases in \cref{3by2}. So, the only way it can fail to have chromatic number 3 is for $Z$ to have loops.  This occurs precisely when $e_i$ is in the $\Z$-span of the columns of $M_Z$ for some $i=1,2,3$.  However, it is straightforward to show that that is not possible. \\

    \item[Case $q=2$]
    Assume without loss of generality that $w_{11}$ and $w_{21}$ are divisible by $3$.
    
    If a row of $M_W$ is divisible by 3, we  apply \cref{lem:4-by-2-three-div-row} to conclude that $M_W$ has loops. The only way for this to have loops is if $w_{11}=w_{21}=0$ and $w_{31}=\pm 1$, in which case $X$ has loops, contrary to assumption.\\
    
    Now suppose that no row of $M_W$ is divisible by 3.  Then $w_{12}$ and $w_{22}$ are $\pm 1\mod 3$. Hence we have\[M_W\xrightarrow{\ocirc}\begin{pmatrix}
    2 & 0\\
    w_{12} +\epsilon_2w_{22}& w_{11}+\epsilon_2w_{21}\\
    w_{32} & w_{31}
    \end{pmatrix},\]
where $\epsilon_2$ is chosen so that $w_{12} +\epsilon_2w_{22}\equiv 0\mod 3$.  By \cref{cor:3-by-2-three-div-row}, the only way the graph associated to the latter matrix can fail to have chromatic number $3$ is for the second column to be $e_3$, so $w_{31}=\pm1$ and $w_{11}=\epsilon_3 w_{21},$ where $\epsilon_3=-\epsilon_2$.  Thus we have that\[M_X\cong\begin{pmatrix}
    1 & 0\\
    1&0\\
    w_{12} & \epsilon_3 w_{21}\\
    w_{22} & w_{21}\\
    w_{32} & \epsilon_4
    \end{pmatrix}
    \cong 
    \begin{pmatrix}
    1&0\\ 1&0\\ 
    w_{12}-\epsilon_3\epsilon_4 w_{21}w_{32} &\epsilon_4\epsilon_3w_{21}\\
    w_{22}-\epsilon_4 w_{21}w_{32}&\epsilon_4 w_{21}\\
    0&1
    \end{pmatrix}.
    \]

Then, the first and third row can be added or subtracted to get
\begin{equation}\label{eq:MX-with-w21-in-2nd-col}\begin{pmatrix}
    1&0\\ 1&0\\ 
    w_{12}-\epsilon_3\epsilon_4 w_{21}w_{32} &\epsilon_4\epsilon_3w_{21}\\
    w_{22}-\epsilon_4 w_{21}w_{32}&\epsilon_4 w_{21}\\
    0&1
    \end{pmatrix}\xrightarrow{\ocirc}\begin{pmatrix}
    \epsilon_5+w_{12}-\epsilon_3\epsilon_4w_{21}w_{32}&\epsilon_4\epsilon_3w_{21} \\
    1&0 \\
    w_{22}-\epsilon_4w_{21}w_{32} &\epsilon_4w_{21}\\
    0&1
    \end{pmatrix}.\end{equation} Because $w_{21}\equiv0\mod 3$, our $\epsilon_5$ can be chosen such that $\epsilon_5+w_{12}\equiv 0\mod 3$, so the top row is divisible by 3. By \cref{lem:4-by-2-three-div-row}, the graph associated to the latter matrix must have loops.  Thus $e_i$ is in the $\Z$-span of its columns for some $i=1,2,3,4$.  From this it follows that either the first column is $e_2$ or the second column is $e_4$.  The first column cannot be $e_2$ because $w_{22}$ is not 0 $\mod 3$.  So the second column is $e_4$, which implies that $w_{21}=0$. But then from (\ref{eq:MX-with-w21-in-2nd-col}) we see that $X$ has loops, a contradiction.  \\

     \item[Case $q=1$] 
Without loss of generality, suppose that $3$ divides $w_{31}$, and that $w_{11}\equiv\epsilon_1\mod 3$ and $w_{21}\equiv\epsilon_2\mod 3$.   Writing the reduction of our matrices mod $3$ rather than the matrices themselves, we have
\[\begin{pmatrix}
    2&0\\w_{12}&\eps_1\\w_{22}&\eps_2\\w_{32}&0
\end{pmatrix}\xrightarrow{\ocirc}\begin{pmatrix}
    2&0\\\eps_1w_{12}-\eps_2 w_{22}&\eps_1^2-\eps_2^2 =0\\w_{32}&0
\end{pmatrix}.\] The contradiction here is now the same as in the $q=3$ case above. \\

    \item[Case $q=0$]Choose $\epsilon_1, \epsilon_2, \epsilon_3$ to be congruent mod $3$ to $w_{11}$, $w_{21}$, and $w_{31}$, respectively.  We then have 
    $$\begin{pmatrix}
    2&0\\w_{12}&\eps_1\\w_{22}&\eps_2\\w_{32}&\eps_3
\end{pmatrix}\xrightarrow{\ocirc}\begin{pmatrix}
    2&0\\ -(\eps_1w_{12}+\eps_2w_{22}) &  -(\epsilon^2_1+\epsilon_2^2)\\\eps_3 w_{32}&1
\end{pmatrix} \equiv \begin{pmatrix}
    2&0\\-(\eps_1w_{12}+\eps_2w_{22}) &1\\\eps_3w_{32}&1
\end{pmatrix}.$$ The argument going forward is similar to that in the $q=3$ and $q=1$ cases.  The matrix on the right satisfies conditions (1)--(3) of MHNF.  By the same logic as in the $q=3$ case, the graph associated to that matrix must have loops.  So $e_i$ is in the $\Z$-span of its columns for some $i=1,2,3$.  We cannot have $i=1$, as the top row is divisible by $2$.  So the second column is $e_2$ or $e_3$. If it were $e_2$, there is a contradiction because $1\not\equiv 0\mod 3$ (from the bottom right entry). If it were $e_3$, there is a contradiction from the middle right entry.   

\end{description}

All possible cases lead to a contradiction, so $\chi(X)<4$. By assumption, $X$ is colorable and not bipartite, so $\chi(X)=3$.
\end{proof}

    \subsection{$m\times 2$ Heuberger matrices with $m>5$}

In this section, prove the main theorem of this paper.

\begin{proof}[Proof of \cref{thm:main}]
The proof is by induction, with \cref{thm:base-case-for-main-thm} as a base case.  Let $m\geq 6$.  We assume that \cref{thm:main} has been proven for $(m-1)\times 2$ matrices, and we now prove it for $m\times 2$ matrices.

Let \[M_X=\begin{pmatrix}
x_{11} & x_{12}\\
\vdots & \vdots\\
x_{m1} & x_{m2}
\end{pmatrix}.\]
Suppose $\chi(X)\geq 4$.  Consider the graph homomorphisms \[M_X\xrightarrow{\ocirc}\begin{pmatrix}
x_{11} & x_{12}\\
\vdots & \vdots\\
x_{m-2,1} & x_{m-2,2}\\
x_{m-1,1}+x_{m1} & x_{m-1,2}+x_{m2}
\end{pmatrix}=M_{W_1}\text{ and }\]\[M_X\xrightarrow{\ocirc}\begin{pmatrix}
x_{11} & x_{12}\\
\vdots & \vdots\\
x_{m-2,1} & x_{m-2,2}\\
x_{m-1,1}-x_{m1} & x_{m-1,2}-x_{m2}
\end{pmatrix}=M_{W_2}\]obtained by adding (respectively, subtracting) the bottom two rows of $M_X$.

By \cref{lem:general-upper-bound} and the inductive hypothesis, it follows that $W_1=M_{W_1}^{SACG}$ and $W_2=M_{W_2}^{SACG}$ both have loops.  Thus we have that \begin{align}\label{eq:system-1}
\alpha_1 y_1 +\beta_1 y_2 &= e_i\\
\label{eq:system-2}\alpha_2 z_1 +\beta_2 z_2 &= e_j
\end{align}for some $i,j=1,\dots,m-1$, where $\alpha_1,\beta_1,\alpha_2,\beta_2$ are integers; $y_1$ and $y_2$ are the first and second columns of $M_{W_1}$, respectively; and $z_1$ and $z_2$ are the first and second columns of $M_{W_2}$, respectively.

Let $r_1, \dots,r_{m-2}$ be the first $m-2$ rows of $M_X$.  These are also the first $m-2$ rows of $M_{W_1}$ as well as of $M_{W_2}$. Note that $m-2\geq 4$.  Hence there exist at least two distinct indices $a,b\notin\{i,j\}$.  Without loss of generality, suppose that $a=1$ and $b=2$.  By (\ref{eq:system-1}) and (\ref{eq:system-2}) we get that several inner products vanish, namely, \[(\alpha_1,\beta_1)\cdot r_1=(\alpha_2,\beta_2)\cdot r_1=0.\]By assumption, $r_1$ is not a zero row.  Moreover, the $i$th (respectively, $j$th) component in (\ref{eq:system-1}) (respectively, (\ref{eq:system-2})) is $1$, so neither $(\alpha_1,\beta_1)$ nor $(\alpha_2,\beta_2)$ is the zero vector and moreover, $\gcd(\alpha_1,\beta_1)=\gcd(\alpha_2,\beta_2)=1$.  It follows that $\alpha_2=\epsilon_1\alpha_1$ and $\beta_2=\epsilon_1\beta_1$ for some $\epsilon_1=\pm 1$.  From the $(m-1)$th components of (\ref{eq:system-1}) and (\ref{eq:system-2}), we then get that \begin{align*}
\alpha_1 (x_{m-1,1}+x_{m1})+\beta_1(x_{m-1,2}+x_{m2}) &=\eta_1\\
\epsilon_1\alpha_1 (x_{m-1,1}-x_{m1})+\epsilon_1\beta_1(x_{m-1,2}-x_{m2}) &=\eta_2,
\end{align*}

where $\eta_1=1$ if $i=m-1$ and $\eta_1=0$ if $i\neq m-1$, and $\eta_1=1$ if $i=m-1$ and $\eta_1=0$ if $i\neq m-1$.  From this we find that \begin{align}\label{eq:last-two-rows-1}
\alpha_1 x_{m-1,1}+\beta_1 x_{m-1,2} &=(\eta_1+\epsilon_1\eta_2)/2\\
\label{eq:last-two-rows-2}\alpha_1 x_{m,1}+\beta_1 x_{m,2} &=(\eta_1-\epsilon_1\eta_2)/2
\end{align}

The right-hand sides of (\ref{eq:last-two-rows-1}) and (\ref{eq:last-two-rows-2}) are integers, so $\eta_1$ and $\eta_2$ are either both $0$ or both $1$.  In the former case, we have that $i<m-1$.  In the latter case, we have that $i=j=m-1$ and that the right-hand side of one of (\ref{eq:last-two-rows-1}) or (\ref{eq:last-two-rows-2}) is $0$, while the other is $1$.  In either case, we get that $\alpha_1 x_1+\beta_1 x_2=e_k$ for some $k=1,\dots,m$, by considering (\ref{eq:system-1}) for the first $m-2$ rows, and (\ref{eq:last-two-rows-1}) and (\ref{eq:last-two-rows-2}) for the last two rows.  But this implies that $X$ has loops, contrary to assumption.\end{proof}

\bibliographystyle{amsplain}
\bibliography{5-by-2}

\end{document}